\documentclass[12pt]{amsart}
\usepackage[top=3cm,bottom=3cm,outer=2cm,inner=2cm,marginpar=1.5cm]{geometry}

\usepackage{amsmath, amssymb, amsfonts, url, xcolor}
\usepackage[colorlinks=true, linkcolor=blue, citecolor=purple, urlcolor=blue]{hyperref}	%hyperref box, link color
\usepackage{newpxtext} %font
\usepackage{verbatim}

\newtheorem{theorem}{Theorem}[section]
\newtheorem{lemma}{Lemma}[section]
\newtheorem{proposition}{Proposition}[section]
\newtheorem{corollary}{Corollary}[section]

\newtheorem*{question*}{Question}

\newtheorem{remark}{Remark}[section]
\newtheorem*{remark*}{Remark}

\numberwithin{equation}{section}

\makeatletter
\newcommand{\ie}{\emph{i.e.}\@ifnextchar.{\!\@gobble}{}}
\newcommand{\eg}{\emph{e.g.}\@ifnextchar.{\!\@gobble}{}}
\newcommand{\etc}{etc\@ifnextchar.{}{.\@}}
\newcommand{\abs}[1]{\lvert#1\rvert}
\makeatother

\newcommand{\R}{\mathbb{R}}
\newcommand{\C}{\mathbb{C}}

\title[Spectrum on Bounded Hyperconvex Domains]{Spectrum of Complete K{\"a}hler Metrics on Bounded Hyperconvex Domains}

\author{Young-Jun Choi}
\address{Department of Mathematics and Institute of Mathematical Science, Pusan National University, Busan, 46241, Republic of Korea}
\email{youngjun.choi@pusan.ac.kr}

\author{Jiwon Brandon Jeong}
\address{Department of Mathematics, Pusan National University, Republic of Korea}
\email{jiwonjeong@pusan.ac.kr}

\date{\today}



\thanks{Corresponding author: Jiwon Brandon Jeong, \texttt{jiwonjeong@pusan.ac.kr}.}

\thanks{This work was supported by the National Research Foundation of Korea (NRF) grant funded by the Korea government (MSIT) (No. 2023R1A2C1007227). The second named author was supported by Basic Science Research Program through the National Research Foundation of Korea (NRF) funded by the Ministry of Education (RS-2026-25581984).}

\begin{document}
	
	\begin{abstract}
		Let $\Omega\subset\C^{n}$ be a bounded domain, and let $\rho:\Omega\to[-1,0)$ be a smooth strictly plurisubharmonic exhaustion function. We consider the logarithmic potential $g=-\log(-\rho)$ and the associated complete K{\"a}hler metric $\omega=dd^{c} g$. We prove that if $\rho$ satisfies the finite weighted Monge--Amp{\`e}re mass condition with weight $(-\rho)^{\varepsilon}$ for every $\varepsilon>0$, then the spectrum and the essential spectrum of the Laplace--Beltrami operator of $(\Omega,\omega)$ both coincide with $[n^{2},\infty)$. In particular, the bottom of the spectrum satisfies $\lambda_{0}(\Delta_{\omega},\Omega)=n^{2}$, and $n^{2}$ is not an $L^{2}$-eigenvalue.
		For each $\alpha > \frac{n}{2}$, we set $f=(-\rho)^{\alpha}$ and prove that $f$ belongs to $W^{1,2}(\Omega,\omega)$ if and only if the Monge--Amp{\`e}re mass of $\rho$ weighted by $(-\rho)^{2\alpha-n}$ is finite.
		The integral identities underlying this equivalence, applied to the complex powers $(-\rho)^{\alpha+it}$, give, after normalization, a singular Weyl sequence at every $n^{2}+4t^{2}$, $t\in\R$, as $\alpha\downarrow\frac{n}{2}$. Combining our theorem with Cegrell's finite-mass exhaustion theorem, we obtain a complete K{\"a}hler metric with spectrum and essential spectrum equal to $[n^{2},\infty)$ on every bounded hyperconvex domain.
	\end{abstract}
	
	\maketitle
	
	\section{Introduction}
	
	Let $(M^{n},\omega)$ be a noncompact complete K{\"a}hler manifold of complex dimension $n$. In local holomorphic coordinates $(z^{1},\ldots,z^{n})$, we write
	\begin{equation*}
		\omega=\sqrt{-1}\sum_{i,j=1}^{n} g_{i\bar{j}}\,dz^{i}\wedge d\bar{z}^{j},
	\end{equation*}
	where $(g_{i\bar{j}})$ is positive definite. We use the following normalization of the nonnegative Laplace--Beltrami operator in local holomorphic coordinates:
	\begin{equation*}
		\Delta_{\omega}
		=
		-4\sum_{i,j=1}^{n} g^{\bar{j} i}
		\frac{\partial^{2}}{\partial z^{i}\partial \bar{z}^{j}},
	\end{equation*}
	where $(g^{\bar{j} i})=(g_{i\bar{j}})^{-1}$.
	Unless otherwise specified, the spaces $C_{0}^{\infty}$, $L^{2}$, and $W^{1,2}$ consist of complex-valued functions.
	We define the bottom of the spectrum of $\Delta_{\omega}$ by
	\begin{equation*}
		\lambda_{0}(\Delta_{\omega},M)
		=
		\inf\left\{
		\frac{4\int_{M}\sum_{i,j=1}^{n} g^{\bar{j} i}f_{i} f_{\bar{j}}\,\omega^{n}}
		{\int_{M}\abs{f}^{2}\,\omega^{n}}
		\;:\;
		0\neq f\in C_{0}^{\infty}(M)
		\right\},
	\end{equation*}
	where the infimum is taken over real-valued functions.
	Here and throughout this paper, we use the convention
	\begin{equation*}
		d^{c}u:=\frac{\sqrt{-1}}{2}(\bar\partial-\partial)u,
		\qquad
		dd^{c} u=\sqrt{-1}\partial\bar\partial u.
	\end{equation*}
	The normalization constant in the Riemannian volume form does not affect the Rayleigh quotient, and therefore we shall use $\omega^{n}$ as the reference volume measure.
	Since the metric is complete, $\Delta_{\omega}$ is essentially self-adjoint on $C_{0}^{\infty}(M)$; see \cite{1973Chernoff}. We use the same notation for its nonnegative self-adjoint closure in $L^{2}(M,\omega^{n})$. We denote its operator domain by $\operatorname{Dom}(\Delta_{\omega})$, its spectrum by $\sigma(\Delta_{\omega})$, and its essential spectrum by $\sigma_{\mathrm{ess}}(\Delta_{\omega})$.
	
	The bottom of the spectrum on noncompact complete K{\"a}hler manifolds is closely related to curvature, potential theory, and boundary geometry. Li and Wang obtained the upper bound $\lambda_{0}\le n^{2}$ under a lower bound on the bisectional curvature \cite{2005LW}. Munteanu later proved the same upper bound under the weaker assumption that the Ricci curvature satisfies $\mathrm{Ric}(\omega)\ge -(n+1)\omega$ \cite{2009Munteanu}. These bounds are optimal, as equality is achieved by the unit ball equipped with the complete K{\"a}hler--Einstein metric with Ricci curvature $-(n+1)$.
	
	On the other hand, Li and Tran computed the bottom of the spectrum for complete K{\"a}hler metrics defined by potentials of the form $-\log(-r)$ on certain bounded pseudoconvex domains. Their work is one of the main starting points of the present paper. We recall the following special case of their results.
	
	\begin{theorem}[Li--Tran \cite{2010LT}]\label{thm:litran_intro}
		Let $\Omega\subset\C^{n}$ be a smoothly bounded strongly pseudoconvex domain, and let $r\in C^{2}(\C^{n})$ be a strictly plurisubharmonic defining function for $\Omega$. Put $u=-\log(-r)$, and let $\Delta_{u}$ denote the Laplace--Beltrami operator associated with the K{\"a}hler metric tensor $(u_{i\bar{j}})$. Then
		\begin{equation*}
			\lambda_{0}(\Delta_{u},\Omega)=n^{2}.
		\end{equation*}
	\end{theorem}
	
	Theorem~\ref{thm:litran_intro} is formulated in terms of a defining function with sufficient regularity up to the boundary. Li and Tran also treated weakly pseudoconvex ellipsoids using exhaustion functions that are smooth inside the domain and H\"older continuous on its closure; see \cite[Section~4]{2010LT}. The purpose of this paper is to obtain the same spectral bottom for smooth strictly plurisubharmonic exhaustion functions satisfying a finite weighted Monge--Amp{\`e}re mass condition, and to determine the full spectrum of the resulting metrics.
	
	Let $\Omega\subset\C^{n}$ be a bounded domain and let
	\begin{equation*}
		\rho:\Omega\to[-1,0)
	\end{equation*}
	be a smooth strictly plurisubharmonic exhaustion function. We set
	\begin{equation*}
		g=-\log(-\rho),
		\qquad
		\omega=dd^{c} g.
	\end{equation*}
	Then $\omega$ is a complete K{\"a}hler metric on $\Omega$. Our main theorem computes the spectrum and the essential spectrum of this complete K{\"a}hler metric.
	
	\begin{theorem}\label{main_thm}
		Let $\Omega\subset\C^{n}$ be a bounded domain, and let $\rho:\Omega\to[-1,0)$ be a smooth strictly plurisubharmonic exhaustion function. Put
		\begin{equation*}
			g=-\log(-\rho),
			\qquad
			\omega=dd^{c} g.
		\end{equation*}
		Assume that
		\begin{equation}\label{eq:weighted_mass_condition}
			\int_{\Omega} (-\rho)^{\varepsilon} (dd^{c} \rho)^{n}<+\infty
			\quad\text{for every }\varepsilon>0.
		\end{equation}
		Then $(\Omega,\omega)$ is complete and
		\begin{equation*}
			\sigma(\Delta_{\omega})=\sigma_{\mathrm{ess}}(\Delta_{\omega})=[n^{2},\infty).
		\end{equation*}
		In particular, $\lambda_{0}(\Delta_{\omega},\Omega)=n^{2}$, and $n^{2}$ is not an $L^{2}$-eigenvalue of $\Delta_{\omega}$. The conclusion also holds if $\rho$ has finite Monge--Amp{\`e}re mass, namely
		\begin{equation*}
			\int_{\Omega}(dd^{c}\rho)^{n}<+\infty.
		\end{equation*}
	\end{theorem}
	
	For the lower bound, we give a direct proof of Li--Tran's estimate by integration by parts. The resulting identity also shows that $n^{2}$ is not an $L^{2}$-eigenvalue; neither assertion requires the weighted Monge--Amp{\`e}re mass condition. For the upper bound, the boundary regularity used in the defining function setting is replaced by the weighted Monge--Amp{\`e}re mass condition \eqref{eq:weighted_mass_condition}. More precisely, for $\alpha>\frac{n}{2}$, define
	\begin{equation*}
		f=(-\rho)^{\alpha}.
	\end{equation*}
	We prove the following equivalence:
	\begin{equation*}
		f\in W^{1,2}(\Omega,\omega)
		\quad\Longleftrightarrow\quad
		\int_{\Omega} (-\rho)^{2\alpha-n}(dd^{c} \rho)^{n}<+\infty.
	\end{equation*}
	For these test functions, we also obtain the exact Rayleigh quotient
	\begin{equation*}
		\frac{4\int_{\Omega} \abs{\partial f}_{\omega}^{2}\,\omega^{n}}
		{\int_{\Omega} f^{2}\,\omega^{n}}
		=2\alpha n.
	\end{equation*}
	Thus \eqref{eq:weighted_mass_condition} allows us to use $f$ for every $\alpha>\frac{n}{2}$, and letting $\alpha\downarrow \frac{n}{2}$ gives the upper bound $\lambda_{0}(\Delta_{\omega},\Omega)\le n^{2}$.
	
	To determine the full spectrum, we consider the complex powers $(-\rho)^{\alpha+it}$, where $t\in\R$. The integral identities for real powers of $-\rho$ show that, after normalization in $L^{2}$, these functions satisfy the eigenvalue equation at $n^{2}+4t^{2}$ up to an error tending to zero in $L^{2}$ as $\alpha\downarrow \frac{n}{2}$. The same identities imply that the normalized functions converge weakly to zero, yielding a singular Weyl sequence. The argument below gives a direct construction under the weighted Monge--Amp{\`e}re mass condition \eqref{eq:weighted_mass_condition}.
	
	The finite-mass case is included in Theorem~\ref{main_thm}: if $\int_{\Omega}(dd^{c}\rho)^{n}<+\infty$, then the weighted Monge--Amp{\`e}re mass condition follows from $0<(-\rho)^{\varepsilon}\le1$.
	
	We now apply Theorem~\ref{main_thm} to bounded hyperconvex domains. A domain $\Omega\subset\C^{n}$ is called \emph{hyperconvex} if it admits a negative plurisubharmonic exhaustion function $\psi\in PSH^{-}(\Omega)$ such that
	\begin{equation*}
		\{z\in\Omega:\psi(z)<c\}\Subset\Omega,
		\quad
		\forall c<0.
	\end{equation*}
	By Corollary~1.3 of Cegrell~\cite{2009Cegrell}, every bounded hyperconvex domain admits a strictly plurisubharmonic exhaustion function in $\mathcal{E}_{0}(\Omega)\cap C^{\infty}(\Omega)$. Related finite-mass exhaustion results in the broader setting of $m$-hyperconvex domains can be found in \cite{2018Ahag}. Since functions in $\mathcal{E}_{0}(\Omega)$ have finite Monge--Amp{\`e}re mass and vanish at the boundary, Theorem~\ref{main_thm} gives the following application.
	
	\begin{corollary}\label{cor:hyperconvex}
		Let $\Omega\subset\C^{n}$ be a bounded hyperconvex domain. Then there exists a smooth strictly plurisubharmonic exhaustion function $\rho:\Omega\to[-1,0)$ such that the complete K{\"a}hler metric
		\begin{equation*}
			\omega=dd^{c}(-\log(-\rho))
		\end{equation*}
		satisfies
		\begin{equation*}
			\sigma(\Delta_{\omega})=\sigma_{\mathrm{ess}}(\Delta_{\omega})=[n^{2},\infty).
		\end{equation*}
		Moreover, $n^{2}$ is not an $L^{2}$-eigenvalue of $\Delta_{\omega}$.
	\end{corollary}
	
	The paper is organized as follows. In Section~\ref{sec2}, we recall Cegrell's finite-mass exhaustion theorem, construct the complete K{\"a}hler metric defined by $g=-\log(-\rho)$, and collect the basic identities and spectral criteria used in the proof. In Section~\ref{sec3}, we establish the integral identities for real powers of $-\rho$ and construct singular Weyl sequences from complex powers. In Section~\ref{sec4}, we prove Theorem~\ref{main_thm} and the hyperconvex corollary.

	\section{Preliminaries}\label{sec2}
	
	In this section, we review the preliminary results needed for the proofs of Theorem~\ref{main_thm} and Corollary~\ref{cor:hyperconvex}. We first recall the consequence of Corollary~1.3 of Cegrell that guarantees the existence of a smooth strictly plurisubharmonic exhaustion function with finite Monge--Amp{\`e}re mass on any bounded hyperconvex domain. We then construct the complete K{\"a}hler metric associated with the potential $g=-\log(-\rho)$ and establish the basic identities used later. We also prove the lower spectral bound and the absence of an $L^{2}$-eigenvalue at $n^{2}$, and recall Weyl's criterion.
	
	Recall that Cegrell's class $\mathcal{E}_{0}$ consists of negative bounded plurisubharmonic functions $u$ on $\Omega$ such that
	\begin{equation*}
		\lim_{z\to\xi}u(z)=0,
		\quad
		\forall\xi\in\partial\Omega,
		\quad
		\text{and}
		\quad
		\int_{\Omega} (dd^{c} u)^{n}<+\infty.
	\end{equation*}
	We refer to \cite{1998Cegrell,2004Cegrell,2009Cegrell} for the basic properties of this class. The following is Corollary~1.3 of Cegrell \cite{2009Cegrell}, stated in the form needed below.
	
	\begin{theorem}[Cegrell \cite{2009Cegrell}]\label{cegrell_thm}
		Let $\Omega\subset\C^{n}$ be a bounded hyperconvex domain. Then there exists a strictly plurisubharmonic exhaustion function
		\begin{equation*}
			\rho\in\mathcal{E}_{0}(\Omega)\cap C^{\infty}(\Omega).
		\end{equation*}
		After multiplying by a sufficiently small positive constant, one may choose $\rho:\Omega\to[-1,0)$ to be a smooth strictly plurisubharmonic exhaustion function satisfying
		\begin{equation*}
			\int_{\Omega} (dd^{c} \rho)^{n}<+\infty.
		\end{equation*}
	\end{theorem}
	
	\begin{remark}
		The normalization $\rho:\Omega\to[-1,0)$ is used only to simplify notation. Multiplication by a positive constant preserves strict plurisubharmonicity and finite Monge--Amp{\`e}re mass.
	\end{remark}
	
	Let now $\Omega\subset\C^{n}$ be a bounded domain and let $\rho:\Omega\to[-1,0)$ be a smooth strictly plurisubharmonic exhaustion function. We set
	\begin{equation*}
		g=-\log(-\rho),
		\qquad
		\omega=dd^{c} g.
	\end{equation*}
	For the derivatives of $\rho$, we write
	\begin{equation*}
		\rho_{i}=\frac{\partial\rho}{\partial z^{i}},
		\qquad
		\rho_{i\bar{j}}=\frac{\partial^{2}\rho}{\partial z^{i}\partial \bar{z}^{j}},
	\end{equation*}
	and denote by $(\rho^{\bar{j} i})$ the inverse matrix of $(\rho_{i\bar{j}})$. We define
	\begin{equation*}
		\abs{\partial\rho}^{2}
		:=
		\sum_{i,j=1}^{n}\rho^{\bar{j} i}\rho_{i}\rho_{\bar{j}}.
	\end{equation*}
	This is the squared norm of $\partial\rho$ with respect to the Hermitian form $dd^{c} \rho$.
	
	The following lemma summarizes the basic properties of the metric defined by $g=-\log(-\rho)$.
	
	\begin{lemma}\label{lem_basic}
		Let $\Omega\subset\C^{n}$ be a bounded domain, and let $\rho:\Omega\to[-1,0)$ be a smooth strictly plurisubharmonic exhaustion function. Set
		\begin{equation*}
			g=-\log(-\rho),
			\qquad
			\omega=dd^{c}g.
		\end{equation*}
		Then the following statements hold:
		\begin{itemize}
			\item[(\romannumeral1)] $g$ is strictly plurisubharmonic, and $\omega$ is a complete K{\"a}hler metric on $\Omega$;
			\item[(\romannumeral2)] If
			\begin{equation*}
				\abs{\partial g}_{\omega}^{2}
				=
				\sum_{i,j=1}^{n}g^{\bar{j}i}g_{i} g_{\bar{j}},
			\end{equation*}
			where $(g^{\bar{j}i})=(g_{i\bar{j}})^{-1}$, then
			\begin{equation}\label{eq:g_norm_bound}
				\abs{\partial g}_{\omega}^{2}<1
				\quad\text{on }\Omega;
			\end{equation}
			\item[(\romannumeral3)] The Monge--Amp{\`e}re measure $\omega^{n}=(dd^{c}g)^{n}$ is given by
			\begin{equation}\label{eq:det_formula}
				\omega^{n}
				=
				(-\rho)^{-n-1}\left(\abs{\partial\rho}^{2}-\rho\right)(dd^{c}\rho)^{n}.
			\end{equation}
		\end{itemize}
	\end{lemma}
	
	\begin{proof}
		Since $\rho<0$, we have
		\begin{equation*}
			\omega=dd^{c}g
			=
			\frac{1}{-\rho}dd^{c}\rho
			+
			\frac{1}{(-\rho)^{2}}d\rho\wedge d^{c}\rho.
		\end{equation*}
		Since $dd^{c}\rho>0$, this shows that $\omega>0$. Hence $g$ is strictly plurisubharmonic and $\omega$ is a K{\"a}hler metric.
		
		The matrix of $\omega$ is
		\begin{equation*}
			g_{i\bar{j}}
			=
			\frac{\rho_{i\bar{j}}}{-\rho}
			+
			\frac{\rho_{i}\rho_{\bar{j}}}{(-\rho)^{2}}.
		\end{equation*}
		By the matrix determinant lemma,
		\begin{align*}
			\det(g_{i\bar{j}})
			&=
			(-\rho)^{-n}
			\det(\rho_{i\bar{j}})
			\left(1+\frac{1}{-\rho}\sum_{i,j=1}^{n}\rho^{\bar{j}i}\rho_{i}\rho_{\bar{j}}\right) \\
			&=
			(-\rho)^{-n-1}
			\left(\abs{\partial\rho}^{2}-\rho\right)
			\det(\rho_{i\bar{j}}).
		\end{align*}
		This proves \eqref{eq:det_formula}.
		
		The inverse matrix formula gives
		\begin{equation}\label{eq:g_norm_exact}
			\abs{\partial g}_{\omega}^{2}
			=
			\frac{\abs{\partial\rho}^{2}}
			{\abs{\partial\rho}^{2}-\rho}.
		\end{equation}
		Since $\rho<0$, the denominator is larger than the numerator. Thus
		\begin{equation*}
			\abs{\partial g}_{\omega}^{2}<1.
		\end{equation*}
		This proves \eqref{eq:g_norm_bound}.
		
		It remains to prove completeness. Since $\rho$ is an exhaustion and $g=-\log(-\rho)$, we have $g(z)\to+\infty$ whenever $z$ leaves every compact subset of $\Omega$.
		
		Let $\gamma:[0,T)\to\Omega$ be a piecewise $C^{1}$ curve which leaves every compact subset of $\Omega$. By \eqref{eq:g_norm_bound} and the Cauchy--Schwarz inequality, there is a constant $C>0$, depending only on our convention for the associated Riemannian metric, such that
		\begin{equation*}
			\abs{\frac{d}{dt}(g\circ\gamma)(t)}
			=
			\abs{dg(\dot\gamma(t))}
			\le
			C\abs{\dot\gamma(t)}_{\omega}
		\end{equation*}
		for almost every $t$. Hence, for every $0<t<T$,
		\begin{align*}
			L_{\omega}\left(\gamma|_{[0,t]}\right)
			&=
			\int_{0}^{t}\abs{\dot\gamma(s)}_{\omega}\,ds \\
			&\ge
			C^{-1}\int_{0}^{t}\abs{\frac{d}{ds}(g\circ\gamma)(s)}\,ds \\
			&\ge
			C^{-1}\abs{g(\gamma(t))-g(\gamma(0))}.
		\end{align*}
		Since $g(\gamma(t))\to+\infty$ as $t\to T$, the length of $\gamma$ is infinite. Therefore every divergent piecewise $C^{1}$ curve has infinite $\omega$-length, and $\omega$ is complete.
	\end{proof}
	
	The lower bound below is a special case of Li--Tran's estimate \cite[Proposition~2.1]{2010LT}. We give a direct proof that also shows that $n^{2}$ is not an $L^{2}$-eigenvalue.
	
	\begin{proposition}\label{prop_lower_noneigen}
		In the setting of Lemma~\ref{lem_basic}, one has
		\begin{equation*}
			\lambda_{0}(\Delta_{\omega},\Omega)\ge n^{2},
		\end{equation*}
		and $n^{2}$ is not an $L^{2}$-eigenvalue of $\Delta_{\omega}$.
	\end{proposition}
	
	\begin{proof}
		By Lemma~\ref{lem_basic}, the metric is complete and $0\le \abs{\partial g}_{\omega}^{2}<1$. Since $\Delta_{\omega}g=-4n$, integration by parts gives
		\begin{equation*}
			\operatorname{Re}\int_{\Omega}f\langle\partial f,\partial g\rangle_{\omega}\,\omega^{n}
			=\frac{1}{8}\int_{\Omega}f^{2}\Delta_{\omega}g\,\omega^{n}
			=-\frac{n}{2}\int_{\Omega}f^{2}\,\omega^{n}
		\end{equation*}
		for every real-valued $f\in C_{0}^{\infty}(\Omega)$. Expanding the square, we obtain
		\begin{align*}
			4\int_{\Omega}\abs{\partial f}_{\omega}^{2}\,\omega^{n}
			-n^{2}\int_{\Omega}f^{2}\,\omega^{n}
			&=4\int_{\Omega}\left|\partial f+\frac{n}{2}f\partial g\right|_{\omega}^{2}\,\omega^{n}\\
			&\quad+n^{2}\int_{\Omega}\bigl(1-\abs{\partial g}_{\omega}^{2}\bigr)f^{2}\,\omega^{n}.
		\end{align*}
		Both terms on the right-hand side are nonnegative. Taking the infimum of the Rayleigh quotient therefore gives $\lambda_{0}(\Delta_{\omega},\Omega)\ge n^{2}$.
		
		Since the metric is complete and $\abs{\partial g}_{\omega}\le1$, the preceding identity extends by density to every real-valued $f\in W^{1,2}(\Omega,\omega)$.
		Suppose that $\Delta_{\omega}f=n^{2}f$ for some nonzero $f\in\operatorname{Dom}(\Delta_{\omega})$. Taking a nonzero real or imaginary part, we may assume that $f$ is real-valued. Then $f\in W^{1,2}(\Omega,\omega)$ and
		\begin{equation*}
			4\int_{\Omega}\abs{\partial f}_{\omega}^{2}\,\omega^{n}
			=n^{2}\int_{\Omega}f^{2}\,\omega^{n}.
		\end{equation*}
		Substituting this equality into the identity above gives
		\begin{equation*}
			\int_{\Omega}\bigl(1-\abs{\partial g}_{\omega}^{2}\bigr)f^{2}\,\omega^{n}=0.
		\end{equation*}
		Since $1-\abs{\partial g}_{\omega}^{2}>0$ everywhere, we obtain $f=0$, a contradiction. Thus $n^{2}$ is not an $L^{2}$-eigenvalue.
	\end{proof}
	
	We recall Weyl's criterion in the following form; see \cite[Lemmas~2.16 and~6.17]{2009Teschl}.
	
	\begin{proposition}[Weyl's criterion]\label{prop_weyl_criterion}
		Let $A$ be a self-adjoint operator on a complex Hilbert space $\mathcal{H}$, and let $\lambda\in\R$. Then $\lambda\in\sigma(A)$ if and only if there exists a sequence $u_{j}\in\operatorname{Dom}(A)$ such that
		\begin{equation*}
			\|u_{j}\|_{\mathcal{H}}=1,
			\qquad
			\|(A-\lambda)u_{j}\|_{\mathcal{H}}\longrightarrow0.
		\end{equation*}
		Moreover, $\lambda\in\sigma_{\mathrm{ess}}(A)$ if and only if such a sequence can be chosen to satisfy
		\begin{equation*}
			u_{j}\rightharpoonup0\quad\text{in }\mathcal{H}.
		\end{equation*}
	\end{proposition}
	
	A sequence $u_{j}\in\operatorname{Dom}(A)$ satisfying the normalization and approximation conditions above is called a \emph{Weyl sequence} for $A$ at $\lambda$. If it also converges weakly to zero, it is called a \emph{singular Weyl sequence}. In particular, if $A$ is bounded below by $a\in\R$ and admits a singular Weyl sequence at every $\lambda\ge a$, then
	\begin{equation*}
		\sigma(A)=\sigma_{\mathrm{ess}}(A)=[a,\infty).
	\end{equation*}
	
	\section{Test functions and Weyl sequences}\label{sec3}
	
	In this section, we first prove the integral identities for the real-valued test functions. This is the step where the finite weighted Monge--Amp{\`e}re mass condition is used. We then use these identities to construct singular Weyl sequences from complex powers of $-\rho$.
	
	Let $\alpha>\frac{n}{2}$ and set
	\begin{equation*}
		f=(-\rho)^{\alpha}.
	\end{equation*}
	
	\begin{proposition}\label{prop_sobolev}
		Let $\Omega\subset\C^{n}$ be a bounded domain, and let $\rho:\Omega\to[-1,0)$ be a smooth strictly plurisubharmonic exhaustion function. Set $g=-\log(-\rho)$ and $\omega=dd^{c} g$. For $\alpha>\frac{n}{2}$, set $f=(-\rho)^{\alpha}$. Then
		\begin{equation*}
			f\in W^{1,2}(\Omega,\omega)
			\quad\Longleftrightarrow\quad
			\int_{\Omega} (-\rho)^{2\alpha-n}(dd^{c} \rho)^{n}<+\infty.
		\end{equation*}
		If these equivalent conditions hold, then
		\begin{equation}\label{eq:L2_identity}
			\int_{\Omega} f^{2}\omega^{n}
			=
			\frac{2\alpha}{2\alpha-n}
			\int_{\Omega} (-\rho)^{2\alpha-n}(dd^{c} \rho)^{n},
		\end{equation}
		and
		\begin{equation}\label{eq:energy_identity}
			4\int_{\Omega} \abs{\partial f}_{\omega}^{2}\omega^{n}
			=
			2\alpha n
			\int_{\Omega} f^{2}\omega^{n}.
		\end{equation}
	\end{proposition}
	
	\begin{proof}
		By \eqref{eq:det_formula},
		\begin{align}
			\int_{\Omega} f^{2}\omega^{n}
			& =
			\int_{\Omega} (-\rho)^{2\alpha-n-1}
			\left(\abs{\partial\rho}^{2}-\rho\right)(dd^{c} \rho)^{n} \nonumber\\
			& =
			\int_{\Omega} (-\rho)^{2\alpha-n-1}\abs{\partial\rho}^{2}(dd^{c} \rho)^{n}
			+
			\int_{\Omega} (-\rho)^{2\alpha-n}(dd^{c} \rho)^{n}. \label{eq:I_decomp}
		\end{align}
		We first express the first term in \eqref{eq:I_decomp} in terms of the second one. A direct computation gives
		\begin{align*}
			dd^{c}\left((-\rho)^{2\alpha-n+1}\right)
			&=
			(2\alpha-n+1)(2\alpha-n)(-\rho)^{2\alpha-n-1}d\rho\wedge d^{c} \rho \\
			&\quad-
			(2\alpha-n+1)(-\rho)^{2\alpha-n}dd^{c} \rho.
		\end{align*}
		Wedge-multiplying this identity by $(dd^{c} \rho)^{n-1}$ and using
		\begin{equation*}
			d\rho\wedge d^{c} \rho\wedge(dd^{c} \rho)^{n-1}
			=
			\frac{1}{n}\abs{\partial\rho}^{2}(dd^{c} \rho)^{n},
		\end{equation*}
		we obtain
		\begin{align*}
			&dd^{c}\left((-\rho)^{2\alpha-n+1}\right)\wedge(dd^{c} \rho)^{n-1} \\
			&\quad =
			\frac{(2\alpha-n+1)(2\alpha-n)}{n}
			(-\rho)^{2\alpha-n-1}\abs{\partial\rho}^{2}(dd^{c} \rho)^{n} \\
			&\qquad-
			(2\alpha-n+1)(-\rho)^{2\alpha-n}(dd^{c} \rho)^{n}.
		\end{align*}
		Let
		\begin{equation*}
			\Omega_{\varepsilon}=
			\{z\in\Omega:-\rho(z)>\varepsilon\},
		\end{equation*}
		where $\varepsilon>0$ is a regular value of $-\rho$. Integrating the preceding wedge-product identity over $\Omega_{\varepsilon}$ gives
		\begin{align}
			&\int_{\Omega_{\varepsilon}}(-\rho)^{2\alpha-n-1}\abs{\partial\rho}^{2}(dd^{c} \rho)^{n} \nonumber\\
			&\quad =
			\frac{n}{(2\alpha-n+1)(2\alpha-n)}
			\int_{\Omega_{\varepsilon}}dd^{c}\left((-\rho)^{2\alpha-n+1}\right)\wedge(dd^{c} \rho)^{n-1} \nonumber\\
			&\qquad+
			\frac{n}{2\alpha-n}
			\int_{\Omega_{\varepsilon}}(-\rho)^{2\alpha-n}(dd^{c} \rho)^{n}. \label{eq:first_term_expression_eps}
		\end{align}
		By Stokes' theorem and $d(dd^{c} \rho)=0$,
		\begin{align}
			&\int_{\Omega_{\varepsilon}}dd^{c}\left((-\rho)^{2\alpha-n+1}\right)\wedge(dd^{c} \rho)^{n-1} \nonumber\\
			&\quad =
			\int_{\partial\Omega_{\varepsilon}}d^{c}\left((-\rho)^{2\alpha-n+1}\right)\wedge(dd^{c} \rho)^{n-1} \nonumber\\
			&\quad =
			-(2\alpha-n+1)\varepsilon^{2\alpha-n}
			\int_{\partial\Omega_{\varepsilon}}d^{c} \rho\wedge(dd^{c} \rho)^{n-1} \nonumber\\
			&\quad =
			-(2\alpha-n+1)\varepsilon^{2\alpha-n}
			\int_{\Omega_{\varepsilon}}(dd^{c} \rho)^{n}. \label{eq:stokes_boundary}
		\end{align}
		Assume now that
		\begin{equation}\label{eq:weighted_mass_alpha}
			\int_{\Omega}(-\rho)^{2\alpha-n}(dd^{c} \rho)^{n}<+\infty.
		\end{equation}
		We claim that the boundary term in \eqref{eq:stokes_boundary} tends to zero, namely
		\begin{equation}\label{eq:boundary_limit_zero}
			\varepsilon^{2\alpha-n}\int_{\Omega_{\varepsilon}}(dd^{c}\rho)^{n}
			\longrightarrow0
			\quad\text{as }\varepsilon\to0.
		\end{equation}
		Set $p=2\alpha-n>0$. Let $\delta>0$. By \eqref{eq:weighted_mass_alpha}, we can choose $\eta>0$ so small that
		\begin{equation*}
			\int_{\{0<-\rho\le\eta\}}(-\rho)^{p}(dd^{c}\rho)^{n}<\frac{\delta}{2}.
		\end{equation*}
		For $0<\varepsilon<\eta$, we split
		\begin{equation*}
			\Omega_{\varepsilon}
			=
			\Omega_{\eta}
			\cup
			\{\varepsilon<-\rho\le\eta\}.
		\end{equation*}
		Since $\Omega_{\eta}\Subset\Omega$, we have
		\begin{equation*}
			\int_{\Omega_{\eta}}(dd^{c}\rho)^{n}<+\infty.
		\end{equation*}
		Thus, after decreasing $\varepsilon>0$ if necessary,
		\begin{equation*}
			\varepsilon^{p}\int_{\Omega_{\eta}}(dd^{c}\rho)^{n}<\frac{\delta}{2}.
		\end{equation*}
		On the remaining part $\{\varepsilon<-\rho\le\eta\}$, we have $\varepsilon^{p}\le(-\rho)^{p}$. Hence
		\begin{align*}
			\varepsilon^{p}\int_{\Omega_{\varepsilon}}(dd^{c}\rho)^{n}
			&\le
			\varepsilon^{p}\int_{\Omega_{\eta}}(dd^{c}\rho)^{n}
			+
			\int_{\{\varepsilon<-\rho\le\eta\}}(-\rho)^{p}(dd^{c}\rho)^{n} \\
			&<
			\frac{\delta}{2}+\frac{\delta}{2}
			=
			\delta.
		\end{align*}
		Since $\delta>0$ was arbitrary, \eqref{eq:boundary_limit_zero} follows.
		
		Letting $\varepsilon\to0$ along regular values in \eqref{eq:first_term_expression_eps}, and using \eqref{eq:stokes_boundary} and \eqref{eq:boundary_limit_zero}, we get
		\begin{equation*}
			\int_{\Omega}(-\rho)^{2\alpha-n-1}\abs{\partial\rho}^{2}(dd^{c} \rho)^{n}
			=
			\frac{n}{2\alpha-n}
			\int_{\Omega}(-\rho)^{2\alpha-n}(dd^{c} \rho)^{n}.
		\end{equation*}
		Substituting this identity into \eqref{eq:I_decomp} gives
		\begin{align*}
			\int_{\Omega}f^{2}\omega^{n}
			&=
			\left(\frac{n}{2\alpha-n}+1\right)
			\int_{\Omega}(-\rho)^{2\alpha-n}(dd^{c} \rho)^{n} \\
			&=
			\frac{2\alpha}{2\alpha-n}
			\int_{\Omega}(-\rho)^{2\alpha-n}(dd^{c} \rho)^{n}.
		\end{align*}
		This proves \eqref{eq:L2_identity}, and in particular the weighted Monge--Amp{\`e}re mass condition \eqref{eq:weighted_mass_alpha} implies $f\in L^{2}(\Omega,\omega^{n})$.
		
		Since $f=e^{-\alpha g}$, we have
		\begin{equation*}
			\partial f=-\alpha f\partial g.
		\end{equation*}
		Using \eqref{eq:g_norm_exact}, we compute
		\begin{align*}
			\int_{\Omega}f^{2}\abs{\partial g}_{\omega}^{2}\omega^{n}
			&=
			\int_{\Omega}(-\rho)^{2\alpha}
			\frac{\abs{\partial\rho}^{2}}{\abs{\partial\rho}^{2}-\rho}
			(-\rho)^{-n-1}
			\left(\abs{\partial\rho}^{2}-\rho\right)(dd^{c} \rho)^{n} \\
			&=
			\int_{\Omega}(-\rho)^{2\alpha-n-1}\abs{\partial\rho}^{2}(dd^{c} \rho)^{n} \\
			&=
			\frac{n}{2\alpha-n}
			\int_{\Omega}(-\rho)^{2\alpha-n}(dd^{c} \rho)^{n}.
		\end{align*}
		Together with \eqref{eq:L2_identity}, this gives
		\begin{align*}
			4\int_{\Omega}\abs{\partial f}_{\omega}^{2}\omega^{n}
			&=
			4\alpha^{2}
			\int_{\Omega}f^{2}\abs{\partial g}_{\omega}^{2}\omega^{n} \\
			&=
			4\alpha^{2}\frac{n}{2\alpha-n}
			\int_{\Omega}(-\rho)^{2\alpha-n}(dd^{c} \rho)^{n} \\
			&=
			2\alpha n
			\int_{\Omega}f^{2}\omega^{n}.
		\end{align*}
		This proves \eqref{eq:energy_identity}. Thus the weighted Monge--Amp{\`e}re mass condition \eqref{eq:weighted_mass_alpha} implies $f\in W^{1,2}(\Omega,\omega)$.
		
		Conversely, if $f\in W^{1,2}(\Omega,\omega)$, then in particular $f\in L^{2}(\Omega,\omega^{n})$. By \eqref{eq:I_decomp},
		\begin{equation*}
			\int_{\Omega}(-\rho)^{2\alpha-n}(dd^{c} \rho)^{n}
			\le
			\int_{\Omega}f^{2}\omega^{n}
			<+\infty.
		\end{equation*}
		This proves the converse implication.
	\end{proof}

	We next construct singular Weyl sequences from complex powers of $-\rho$. In the rest of the paper, $\|\cdot\|_{2}$ denotes the norm in $L^{2}(\Omega,\omega^{n})$.
	
	\begin{lemma}[Singular Weyl sequences]\label{lem_weyl_sequence}
		In the setting of Proposition~\ref{prop_sobolev}, assume the weighted Monge--Amp{\`e}re mass condition \eqref{eq:weighted_mass_condition}. For $\alpha>\frac{n}{2}$, set
		\begin{equation*}
			I_{\alpha}=\int_{\Omega}e^{-2\alpha g}\,\omega^{n}.
		\end{equation*}
		Fix $t\in\R$ and put $\lambda_{t}=n^{2}+4t^{2}$. For any sequence $\alpha_{j}>\frac{n}{2}$ with $\alpha_{j}\downarrow\frac{n}{2}$, the functions
		\begin{equation*}
			u_{j}=\frac{e^{-(\alpha_{j}+it)g}}{\sqrt{I_{\alpha_{j}}}}
			=\frac{(-\rho)^{\alpha_{j}+it}}{\sqrt{I_{\alpha_{j}}}}
		\end{equation*}
		form a singular Weyl sequence for $\Delta_{\omega}$ at $\lambda_{t}$. Here $(-\rho)^{\alpha_{j}+it}:=\exp((\alpha_{j}+it)\log(-\rho))$, with the real logarithm. More precisely,
		\begin{equation*}
			\begin{gathered}
				u_{j}\in\operatorname{Dom}(\Delta_{\omega}),
				\qquad \|u_{j}\|_{2}=1,\\
				\|(\Delta_{\omega}-\lambda_{t})u_{j}\|_{2}\longrightarrow0,
				\qquad u_{j}\rightharpoonup0
				\quad\text{in }L^{2}(\Omega,\omega^{n}).
			\end{gathered}
		\end{equation*}
	\end{lemma}
	
	\begin{proof}
		Put $s_{j}=\alpha_{j}+it$ and $F_{j}=e^{-s_{j}g}$. The weighted Monge--Amp{\`e}re mass condition and Proposition~\ref{prop_sobolev} give $0<I_{\alpha_{j}}<+\infty$. Since $\abs{F_{j}}^{2}=e^{-2\alpha_{j}g}$, we have $F_{j}\in L^{2}$ and $\|u_{j}\|_{2}=1$. By \eqref{eq:g_norm_exact}, $0\le \abs{\partial g}_{\omega}^{2}<1$.
		
		A direct computation of the differential expression gives
		\begin{equation}\label{eq:complex_power_laplacian}
			\Delta_{\omega}F_{j}=4s_{j}\bigl(n-s_{j}\abs{\partial g}_{\omega}^{2}\bigr)F_{j}.
		\end{equation}
		Since $\abs{\partial g}_{\omega}^{2}$ is bounded, the right-hand side belongs to $L^{2}$. By Lemma~\ref{lem_basic}, $(\Omega,\omega)$ is complete, so the Laplacian on $C_{0}^{\infty}(\Omega)$ is essentially self-adjoint. Hence $F_{j}\in\operatorname{Dom}(\Delta_{\omega})$, and \eqref{eq:complex_power_laplacian} holds for the self-adjoint operator. Consequently, $u_{j}\in\operatorname{Dom}(\Delta_{\omega})$.
		
		Applying \eqref{eq:energy_identity} to the real-valued function $f_{j}=e^{-\alpha_{j}g}$ and using $\partial f_{j}=-\alpha_{j}f_{j}\partial g$, we obtain
		\begin{equation*}
			4\alpha_{j}^{2}\int_{\Omega}\abs{\partial g}_{\omega}^{2}\,e^{-2\alpha_{j}g}\,\omega^{n}
			=2\alpha_{j}nI_{\alpha_{j}}.
		\end{equation*}
		Thus
		\begin{equation*}
			\int_{\Omega}\bigl(1-\abs{\partial g}_{\omega}^{2}\bigr)\abs{u_{j}}^{2}\,\omega^{n}
			=\frac{2\alpha_{j}-n}{2\alpha_{j}},
		\end{equation*}
		and, since $\bigl(1-\abs{\partial g}_{\omega}^{2}\bigr)^{2}\le1-\abs{\partial g}_{\omega}^{2}$,
		\begin{equation*}
			\|\bigl(1-\abs{\partial g}_{\omega}^{2}\bigr)u_{j}\|_{2}^{2}
			\le\frac{2\alpha_{j}-n}{2\alpha_{j}}.
		\end{equation*}
		By \eqref{eq:complex_power_laplacian},
		\begin{equation*}
			(\Delta_{\omega}-\lambda_{t})u_{j}
			=\left[4s_{j}(n-s_{j})-\lambda_{t}+4s_{j}^{2}\bigl(1-\abs{\partial g}_{\omega}^{2}\bigr)\right]u_{j}.
		\end{equation*}
		Therefore,
		\begin{equation*}
			\|(\Delta_{\omega}-\lambda_{t})u_{j}\|_{2}
			\le\abs{4s_{j}(n-s_{j})-\lambda_{t}}
			+4\abs{s_{j}}^{2}\sqrt{\frac{2\alpha_{j}-n}{2\alpha_{j}}}
			\longrightarrow0
		\end{equation*}
		as $j\to\infty$, since $s_{j}\to\frac{n}{2}+it$ and $4s_{j}(n-s_{j})\to\lambda_{t}$.
		
		It remains to prove that $u_{j}\rightharpoonup0$. By \eqref{eq:L2_identity},
		\begin{equation*}
			I_{\alpha_{j}}
			=\frac{2\alpha_{j}}{2\alpha_{j}-n}
			\int_{\Omega}(-\rho)^{2\alpha_{j}-n}(dd^{c}\rho)^{n}.
		\end{equation*}
		For all sufficiently large $j$, we have $0<2\alpha_{j}-n\le1$, and hence
		\begin{equation*}
			I_{\alpha_{j}}
			\ge\frac{2\alpha_{j}}{2\alpha_{j}-n}
			\int_{\Omega}(-\rho)(dd^{c}\rho)^{n}
			\longrightarrow+\infty.
		\end{equation*}
		The last integral is finite by \eqref{eq:weighted_mass_condition} and positive because $\rho$ is strictly plurisubharmonic. Since $g\ge0$, for every compact set $K\subset\Omega$,
		\begin{equation*}
			\int_{K}\abs{u_{j}}^{2}\,\omega^{n}
			=\frac{1}{I_{\alpha_{j}}}\int_{K}e^{-2\alpha_{j}g}\,\omega^{n}
			\le\frac{1}{I_{\alpha_{j}}}\int_{K}\omega^{n}
			\longrightarrow0.
		\end{equation*}
		Together with $\|u_{j}\|_{2}=1$, this local convergence implies weak convergence to zero in $L^{2}$. All the conditions for a singular Weyl sequence are therefore satisfied.
	\end{proof}

	\section{Proof of the main theorem}\label{sec4}
	
	The main theorem follows from the lower estimate and Lemma~\ref{lem_weyl_sequence}.
	
	\begin{proof}[Proof of Theorem~\ref{main_thm}]
		By Lemma~\ref{lem_basic}, the metric $\omega=dd^{c}(-\log(-\rho))$ is complete. Proposition~\ref{prop_lower_noneigen} gives
		\begin{equation*}
			\lambda_{0}(\Delta_{\omega},\Omega)\ge n^{2}.
		\end{equation*}
		It also shows that $n^{2}$ is not an $L^{2}$-eigenvalue. Hence
		\begin{equation*}
			\sigma(\Delta_{\omega})\subset[n^{2},\infty).
		\end{equation*}
		
		For each $t\in\R$, Lemma~\ref{lem_weyl_sequence} provides a singular Weyl sequence for $\Delta_{\omega}$ at $\lambda_{t}=n^{2}+4t^{2}$. Proposition~\ref{prop_weyl_criterion} therefore gives
		\begin{equation*}
			n^{2}+4t^{2}\in\sigma_{\mathrm{ess}}(\Delta_{\omega}).
		\end{equation*}
		As $t$ ranges over $\R$, the numbers $n^{2}+4t^{2}$ fill $[n^{2},\infty)$. Consequently,
		\begin{equation*}
			[n^{2},\infty)
			\subset\sigma_{\mathrm{ess}}(\Delta_{\omega})
			\subset\sigma(\Delta_{\omega})
			\subset[n^{2},\infty).
		\end{equation*}
		This proves
		\begin{equation*}
			\sigma(\Delta_{\omega})=\sigma_{\mathrm{ess}}(\Delta_{\omega})=[n^{2},\infty),
		\end{equation*}
		and in particular $\lambda_{0}(\Delta_{\omega},\Omega)=n^{2}$.
		
		The finite-mass assertion follows immediately from $0<(-\rho)^{\varepsilon}\le1$ for every $\varepsilon>0$.
	\end{proof}

	\begin{proof}[Proof of Corollary~\ref{cor:hyperconvex}]
		By Theorem~\ref{cegrell_thm}, there exists a smooth strictly plurisubharmonic exhaustion function $\rho:\Omega\to[-1,0)$ such that
		\begin{equation*}
			\int_{\Omega} (dd^{c} \rho)^{n}<+\infty.
		\end{equation*}
		The finite-mass assertion in Theorem~\ref{main_thm} applied to this $\rho$ gives the desired complete K{\"a}hler metric.
	\end{proof}
	
	We finish with two remarks concerning the curvature of the constructed metric and an extension to hyperconvex manifolds under the same weighted Monge--Amp{\`e}re mass condition.
	
	\begin{remark}[Ricci form]
		For the metric $\omega=dd^{c}(-\log(-\rho))$, one has
		\begin{equation*}
			\mathrm{Ric}(\omega)
			=
			-(n+1)\omega
			-
			dd^{c}\log J(\rho),
		\end{equation*}
		where
		\begin{equation*}
			J(\rho)
			=
			\left(\abs{\partial\rho}^{2}-\rho\right)\det(\rho_{i\bar{j}}).
		\end{equation*}
		Therefore the extra term in the K{\"a}hler--Einstein equation is
		$-dd^{c}\log J(\rho)$. The proof of Theorem~\ref{main_thm} does not
		require an estimate on this term. It would be interesting to find
		additional conditions on $\rho$ which make this extra term small near the
		boundary.
	\end{remark}
	
	\begin{remark}[Hyperconvex manifolds]
		There is also a natural manifold version of the present question. Following
		Chen \cite{2004Chen}, a complex manifold $M$ is called hyperconvex if there
		exists a strictly plurisubharmonic function $\rho:M \to [-1,0)$ such that $\{x\in M:\rho(x)<c\}\Subset M$ for every $c<0$. Chen proved
		that every hyperconvex manifold has a complete Bergman metric.
		
		In this terminology, the condition used in this paper suggests the following
		more specific class of hyperconvex manifolds: those admitting a smooth
		strictly plurisubharmonic exhaustion $\rho:M\to[-1,0)$ such that
		\begin{equation*}
			\int_{M} (-\rho)^{\varepsilon}(dd^{c}\rho)^{n}<+\infty
			\quad\text{for every }\varepsilon>0.
		\end{equation*}
		For such an exhaustion, the same construction gives a complete K{\"a}hler
		metric $\omega=dd^{c}(-\log(-\rho))$.
		Moreover, since the proof above only uses the exhaustion property of
		$\rho$, the estimate $\abs{\partial(-\log(-\rho))}_{\omega}^{2}<1$, and
		the finite weighted Monge--Amp{\`e}re mass condition, the same argument
		gives
		\begin{equation*}
			\sigma(\Delta_{\omega})=\sigma_{\mathrm{ess}}(\Delta_{\omega})=[n^{2},\infty).
		\end{equation*}
		Proposition~\ref{prop_lower_noneigen} also applies in this setting, so $n^{2}$ is not an $L^{2}$-eigenvalue.
		It would be interesting to know which hyperconvex manifolds admit an
		exhaustion satisfying the finite weighted Monge--Amp{\`e}re mass condition.
	\end{remark}
	
	\section*{Acknowledgements}
	
	ChatGPT (OpenAI) was used to assist with mathematical discussions and manuscript preparation.
	The authors reviewed the AI-assisted material and take full responsibility for the paper.
	
	%\bibliographystyle{plain}
	%\bibliography{ref}

\begin{thebibliography}{99}
		
		\bibitem{2018Ahag}
		Per {\AA}hag, Rafa\l{} Czy\.z, and Lisa Hed.
		\newblock The geometry of {$m$}-hyperconvex domains.
		\newblock {\em J. Geom. Anal.}, 28(4):3196--3222, 2018.
		
		\bibitem{1998Cegrell}
		Urban Cegrell.
		\newblock Pluricomplex energy.
		\newblock {\em Acta Math.}, 180(2):187--217, 1998.
		
		\bibitem{2004Cegrell}
		Urban Cegrell.
		\newblock The general definition of the complex {M}onge-{A}mp\`ere operator.
		\newblock {\em Ann. Inst. Fourier (Grenoble)}, 54(1):159--179, 2004.
		
		\bibitem{2009Cegrell}
		Urban Cegrell.
		\newblock Approximation of plurisubharmonic functions in hyperconvex domains.
		\newblock In {\em Complex analysis and digital geometry}, volume~86 of {\em
			Acta Univ. Upsaliensis Skr. Uppsala Univ. C Organ. Hist.}, pages 125--129.
		Uppsala Universitet, Uppsala, 2009.
		
		\bibitem{2004Chen}
		Bo-Yong Chen.
		\newblock Bergman completeness of hyperconvex manifolds.
		\newblock {\em Nagoya Math. J.}, 175:165--170, 2004.
		
		\bibitem{1973Chernoff}
		Paul~R. Chernoff.
		\newblock Essential self-adjointness of powers of generators of hyperbolic equations.
		\newblock {\em J. Funct. Anal.}, 12(4):401--414, 1973.
		
		\bibitem{2005LW}
		Peter Li and Jiaping Wang.
		\newblock Comparison theorem for {K}\"ahler manifolds and positivity of
		spectrum.
		\newblock {\em J. Differential Geom.}, 69(1):43--74, 2005.
		
		\bibitem{2010LT}
		Song-Ying Li and My-An Tran.
		\newblock Infimum of the spectrum of {L}aplace-{B}eltrami operator on a bounded
		pseudoconvex domain with a {K}\"ahler metric of {B}ergman type.
		\newblock {\em Comm. Anal. Geom.}, 18(2):375--395, 2010.
		
		\bibitem{2009Munteanu}
		Ovidiu Munteanu.
		\newblock A sharp estimate for the bottom of the spectrum of the {L}aplacian on
		{K}\"ahler manifolds.
		\newblock {\em J. Differential Geom.}, 83(1):163--187, 2009.
		
		\bibitem{2009Teschl}
		Gerald Teschl.
		\newblock {\em Mathematical methods in quantum mechanics: With applications to
			{S}chr\"odinger operators}, volume~99 of {\em Graduate Studies in Mathematics}.
		\newblock American Mathematical Society, Providence, RI, 2009.
		
	\end{thebibliography}

\end{document}